\pdfoutput=1
\documentclass[12pt]{article}
\usepackage[margin=30mm]{geometry}
\usepackage[T1]{fontenc}
\usepackage{indentfirst,graphicx}
\usepackage[UKenglish]{babel}
\usepackage[usenames,dvipsnames,svgnames,table]{xcolor}
\usepackage{amsfonts,amsmath,amssymb,amsbsy,amsthm,makeidx}
\allowdisplaybreaks
\usepackage{hyperref}
\hypersetup{ 
unicode=true,
colorlinks=true,
linkcolor={black},
citecolor={black},
urlcolor={blue!60!black},
pdftitle={Clustered Coloring of Graphs with Bounded Layered Treewidth and Bounded Degree},
pdfauthor={Chun-Hung Liu and David~R.~Wood}}
\usepackage{enumitem}
\setlist[itemize]{itemsep=0ex}
\setlist[enumerate]{itemsep=0ex}
\setlist[description]{itemsep=0ex}
\usepackage[longnamesfirst,numbers,sort&compress]{natbib}
\makeatletter
\def\NAT@spacechar{~}
\makeatother
\usepackage[noabbrev,capitalise]{cleveref}
\crefname{lem}{Lemma}{Lemmas}
\crefname{thm}{Theorem}{Theorems}
\crefname{prop}{Proposition}{Propositions}
\crefname{obs}{Observation}{Observations}
\crefname{conj}{Conjecture}{Conjectures}
\crefname{claim}{Claim}{Claims}
\crefformat{equation}{(#2#1#3)}
\Crefformat{equation}{Equation #2(#1)#3}
\newtheorem{theorem}{Theorem}
\newtheorem{lemma}[theorem]{Lemma}

\newtheorem{corollary}[theorem]{Corollary}

\makeatletter
\renewcommand\section{\@startsection {section}{1}{\z@}{-3ex \@plus -1ex \@minus -.2ex}{2ex \@plus.2ex}{\normalfont\large\bfseries}}
\renewcommand\subsection{\@startsection{subsection}{2}{\z@}{-2.5ex\@plus -1ex \@minus -.2ex}{1.5ex \@plus .2ex}{\normalfont\normalsize\bfseries}}
\renewcommand\subsubsection{\@startsection{subsubsection}{3}{\z@}{-2ex\@plus -1ex \@minus -.2ex}{1ex \@plus .2ex}{\normalfont\normalsize\bfseries}}
 \renewcommand\paragraph{\@startsection{paragraph}{4}{\z@}{1.5ex \@plus.5ex \@minus.2ex}{-1em}{\normalfont\normalsize\bfseries}}
\renewcommand\subparagraph{\@startsection{subparagraph}{5}{\parindent}  {1.5ex \@plus.5ex \@minus .2ex}  {-1em} {\normalfont\normalsize\bfseries}}
\makeatother
\def\X {{\mathcal X}}

\def\C {{\mathcal C}}

\renewcommand{\geq}{\geqslant}
\renewcommand{\leq}{\leqslant}

\renewcommand{\thefootnote}{\fnsymbol{footnote}}	
\newcommand{\defn}[1]{\textcolor{Maroon}{\emph{#1}}\index{#1}}
\newcommand{\mathdef}[2]{\textcolor{Maroon}{\emph{#1-#2}}\index{#2@#1-#2}}

\begin{document}
\title{\bf\Large\boldmath 
Clustered Coloring of Graphs with \\
Bounded Layered Treewidth and Bounded Degree\footnote{This material is based upon work supported by the National Science Foundation under Grant No.\ DMS-1664593, DMS-1929851, DMS-1954054 and DMS-2144042.}}
\footnotetext[0]{First released on 22nd May 2019 as part of arXiv:1905.08969. Revised \today.}
\author{%
Chun-Hung Liu\footnote{Department of Mathematics, Texas A\&M University, Texas, USA, \texttt{chliu@math.tamu.edu}. Partially supported by NSF under award No.\ DMS-1664593, DMS-1929851 and DMS-1954054 and CAREER award DMS-2144042.}
\quad David R. Wood\footnote{School of Mathematics, Monash University, Melbourne, Australia, \texttt{david.wood@monash.edu}. Research supported by the Australian Research Council.}
}
\date{}
\maketitle
\begin{abstract}
The clustering of a graph coloring is the maximum size of monochromatic components.
This paper studies colorings with bounded clustering in graph classes with bounded layered treewidth, which include planar graphs, graphs of bounded Euler genus, graphs embeddable on a fixed surface with a bounded number of crossings per edge, map graphs, amongst other examples. Our main theorem says that every graph with layered treewidth at most $k$ and with maximum degree at most $\Delta$ is $3$-colorable with clustering $O(k^{19}\Delta^{37})$.
This is the first known polynomial bound on the clustering. This greatly improves upon a corresponding result of Esperet and Joret for graphs of bounded genus. 
\end{abstract}
\renewcommand{\thefootnote}{\arabic{footnote}}	
\section{Introduction}
\label{Introduction}

This paper considers graph\footnote{Let $G$ be a graph with vertex-set $V(G)$ and edge-set $E(G)$. 
For $v\in V(G)$, let $N_G(v):=\{w\in V(G): vw\in E(G)\}$ be the neighborhood of $v$, and let  $N_G[v] :=N_G(v) \cup\{v\}$. 
For $X\subseteq V(G)$, let $N_G(X):= \bigcup_{v\in X} (N_G(v) \setminus X)$ and $N_G[X]:= N_G(X)\cup X$.
Denote the subgraph of $G$ induced by $X$ by $G[X]$.
Let $\mathbb{N}:=\{1,2,\dots\}$ and $\mathbb{N}_0:=\{0,1,2,\dots\}$. 
For $m,n\in\mathbb{N}_0$, let $[m,n]:=\{m,m+1,\dots,n\}$ and $[n]:=[1,n]$.} colorings where the condition that adjacent vertices are assigned distinct colors is relaxed. Instead, we require that every monochromatic component has bounded size (for a given graph class). 
More formally, a \defn{coloring} of a graph $G$ is a function that assigns one color to each vertex of $G$. For a coloring $c$ of $G$, a \mathdef{$c$}{monochromatic component}, or simply \defn{monochromatic component}, 
is any connected component of the subgraph of $G$ induced by all the vertices assigned the same color. A coloring has \defn{clustering} $\eta$ if every monochromatic component has at most $\eta$ vertices. Our focus is on minimizing the number of colors, with small monochromatic components as a secondary goal. 
The \defn{clustered chromatic number} of a graph class $\mathcal{G}$ is the minimum integer $k$ such that for some integer $\eta$, every graph in $\mathcal{G}$ is $k$-colorable with clustering $\eta$. There have been several recent papers on this topic~\citep{NSSW19,vdHW18,HST03,ADOV03,CE19,Kawa08,KM07,LMST08,KO19,EJ14,EO16,DN17,LO17,HW19,MRW17}; see \citep{WoodSurvey} for a survey. 

Planar graphs are 4-colorable, and indeed have clustered chromatic number 4. That is, for every $c$ there is a planar graph that is not 3-colorable with clustering $c$ \cite{KMRV97}. These examples have unbounded maximum degree. This led \citet{KMRV97} to conjecture that planar graphs with bounded maximum degree are 3-colorable with bounded clustering. Note that three colors is best possible since the Hex Lemma~\citep{Gale79} says that every 2-coloring of the $n\times n$ planar triangular grid (which has  maximum degree 6) contains a monochromatic path of length at least $n$. 

\citet{EJ14} proved the conjecture of \citet{KMRV97}. In particular, they proved that every planar graph with maximum degree $\Delta$ is 3-colorable with clustering  $\Delta^{O(\Delta)}$. While \citet{EJ14} made no effort to reduce this function, their method will not lead to a  sub-exponential clustering bound. We prove the first known polynomial bound in such a result. \citet{EJ14} in fact proved 3-colorability in the more general setting of graphs embeddable on a fixed surface\footnote{The \defn{Euler genus} of an orientable surface with $h$ handles is $2h$. The \defn{Euler genus} of  a non-orientable surface with $c$ cross-caps is $c$. The \defn{Euler genus} of a graph $G$ is the minimum Euler genus of a surface in which $G$ embeds (with no crossings).}. In particular, they proved that graphs with Euler genus $g$ are 3-colorable with clustering $\Delta^{O(\Delta\,2^g)}$. We prove a polynomial bound in this setting as well. 

\begin{theorem}\label{GenusDegree}
Every graph with Euler genus at most $g$ and maximum degree at most $\Delta$ is 3-colorable with clustering $O(g^{19}\Delta^{37})$. 
\end{theorem}

Note that \citet{LO17} proved that for all fixed $t,\Delta$ there exists a constant $c$ such that every (odd) $K_t$-minor-free graph with maximum degree at most $\Delta$ is 3-colorable with clustering $c$.  
This result extends \cref{GenusDegree} to more general graph classes but with much worse clustering $c$. 

\cref{GenusDegree} is in fact a corollary of a more general result of this paper in terms of layered treewidth. First we explain what this means. A \defn{tree-decomposition} of a graph $G$ is a pair $(T,\mathcal{X} = (X_x:x\in V(T))$, where $T$ is a tree,  and for each node $x\in V(T)$,  $X_x$ is a non-empty subset of $V(G)$ called a \defn{bag}, such that for each vertex $v\in V(G)$, the set $\{x\in V(T):v\in X_x\}$ induces a non-empty (connected) subtree of $T$, and for each edge $vw\in E(G)$ there is a node $x\in V(T)$ such that $\{v,w\}\subseteq X_x$. 
The \defn{width} of a tree-decomposition $(T,\mathcal{X})$ is $\max\{|X_x|-1: x\in V(T)\}$. The \defn{treewidth} of a graph $G$ is the minimum width of a tree-decomposition of $G$. Treewidth is a key parameter in algorithmic and structural graph theory; see \citep{Reed03,Reed97,Bodlaender-TCS98,HW17} for surveys.

A \defn{layering} of a graph $G$ is an ordered partition $(V_1,\dots,V_n)$ of $V(G)$ into (possibly empty) sets such that for each edge $vw\in E(G)$ there exists $i\in[1,n-1]$ such that $\{v,w\}\subseteq V_i\cup V_{i+1}$. 
The \defn{layered treewidth} of a graph $G$ is the minimum nonnegative integer $\ell$ such that $G$ has a tree-decomposition $(T, \mathcal{X} = (X_x:x\in V(T))$ and a layering $(V_1,\dots,V_n)$, such that $|X_x\cap V_i|\leq\ell$ for each bag $X_x$ and layer $V_i$. This says that the subgraph induced by each layer has bounded treewidth, and moreover, a single tree-decomposition of $G$ has bounded treewidth when restricted to each layer. In fact, these properties hold when considering a bounded sequence of consecutive layers. 

Layered treewidth was independently introduced by \citet{DMW17} and \citet{Shahrokhi13}. \citet{DMW17} proved that every planar graph has layered treewidth at most 3;  more generally, that every graph with Euler genus at most $g$ has layered treewidth at most $2g+3$; and most generally, that a minor-closed class has bounded layered treewidth if and only if it excludes some apex graph as a minor. Layered treewidth is of interest beyond minor-closed classes, since as described below, there are several natural graph classes that have bounded layered treewidth but contain arbitrarily large complete graph minors. 

We prove that only three colors are needed
for clustered coloring of graphs with bounded layered treewidth and with bounded maximum degree. This is the main result of the paper. 

\begin{theorem}\label{ltwdegree}
Every graph with layered treewidth at most $w$ and maximum degree at most $\Delta$ is 3-colorable with clustering $O(w^{19}\Delta^{37})$.
\end{theorem}

\cref{GenusDegree} is an immediate corollary of \cref{ltwdegree} and the above-mentioned bound of \citet{DMW17}. The proof of \cref{ltwdegree} (presented in \cref{BoundedDegree}) is simpler than the proof in \cite{EJ14}, avoiding many technicalities that arise when dealing with graph embeddings. This proof highlights the utility of layered treewidth as a general tool. Moreover, it leads to more general results, as we now explain. 

\medskip
We now give three examples of graph classes with bounded layered treewidth, for which \cref{ltwdegree} give interesting results. 

\paragraph{\boldmath $(g,k)$-Planar Graphs:} A graph is \mathdef{$(g,k)$}{planar} if it can be drawn in a surface of Euler genus at most $g$ with at most $k$ crossings on each edge (assuming no three edges cross at a single point). Such graphs can contain arbitrarily large complete graph minors, even in the $g=0$ and $k=1$ case \cite{DEW17}, so the above-mentioned result of \citet{LO17} is not applicable. On the other hand, \citet{DEW17} proved that every $(g,k)$-planar graph has layered treewidth at most $(4g + 6)(k + 1)$. \cref{ltwdegree} then implies:

\begin{corollary}
\label{gkPlanarDegree}
For all $g,k,\Delta \in \mathbb{N}$, every $(g,k)$-planar graph with maximum degree at most $\Delta$ is 3-colorable with clustering $O(g^{19}k^{19}\Delta^{37})$. 
\end{corollary}

\paragraph{Map Graphs:} Map graphs are defined as follows. Start with a graph $G_0$ embedded in a surface of Euler genus $g$, with each face labelled a ``nation'' or a ``lake'', where each vertex of $G_0$ is incident with at most $d$ nations. Let $G$ be the graph whose vertices are the nations of $G_0$, where two vertices are adjacent in $G$ if the corresponding faces in $G_0$ share a vertex. Then $G$ is called a \mathdef{$(g,d)$}{map graph}. A $(0,d)$-map graph is called a \emph{(plane)} \mathdef{$d$}{map graph}; such graphs have been extensively studied \citep{FLS-SODA12,Chen07,DFHT05,CGP02,Chen01}. The $(g,3)$-map graphs are precisely the graphs of Euler genus at most $g$ (see \citep{CGP02,DEW17}). 
So $(g,d)$-map graphs provide a natural generalization of graphs embedded in a surface that allows for arbitrarily large cliques even in the $g=0$ case (since if a vertex of $G_0$ is incident with $d$ nations then $G$ contains $K_d$). 
\citet{DEW17} proved that every $(g,d)$-map graph has layered treewidth at most $(2g+3)(2d+1)$. Thus \cref{ltwdegree} implies:

\begin{corollary}
\label{MapGraphBoundedDegree}
For all $g,d,\Delta \in \mathbb{N}$, every $(g,d)$-map graph with maximum degree at most $\Delta$ is 3-colorable with clustering $O(g^{19}d^{19}\Delta^{37})$.
\end{corollary}

\paragraph{String Graphs:} A \defn{string graph} is the intersection graph of a set of curves in the plane with no three curves meeting at a single point \cite{PachToth-DCG02,SS-JCSS04,SSS-JCSS03,Krat-JCTB91,FP10,FP14}. For an integer $k\geq 2$, if each curve is in at most $k$ intersections with other curves, then the corresponding string graph is called a \mathdef{$k$}{string graph}. Note that two curves can intersect multiple times and contribute more than one intersection to each curve. A \mathdef{$(g,k)$}{string} graph is defined analogously for curves on a surface of Euler genus at most $g$. \citet{DJMNW18} proved that every $(g,k)$-string graph has layered treewidth at most $2(k-1)(2g+3)$. By definition, the maximum degree of a $(g,k)$-string graph is at most $k$. Thus \cref{ltwdegree} implies:

\begin{corollary}
\label{StringColor}
For all integers $g\geq 0$ and $k\geq 2$, there exists $\eta\in\mathbb{N}$ such that every $(g,k)$-string graph is $3$-colorable with clustering $O(g^{19}k^{56})$.
\end{corollary}

\section{The Proof}
\label{BoundedDegree}

This section proves \cref{ltwdegree}, which says that graphs of bounded layered treewidth and bounded maximum degree are  3-colorable with bounded clustering. We need the following analogous result by \citet{ADOV03} for bounded treewidth graphs. 

\begin{lemma}[\citep{ADOV03}]
\label{ClusteringDegreeTreewidth}
There is a function $f:\mathbb{N}\times\mathbb{N}\to\mathbb{N}$ such that every graph with treewidth $w$  and maximum degree $\Delta$ is 2-colorable with clustering $f(w,\Delta) \in O(w\Delta)$. 
\end{lemma}

For a graph $G$ with bounded maximum degree and bounded layered treewidth, if $(V_1,\dots,V_n)$ is the corresponding layering of $G$, then 
\cref{ClusteringDegreeTreewidth} is applicable to $G[V_i]$, which has bounded treewidth. 
The idea of the proof of \cref{ltwdegree} is to 
use colors 1 and 2 for all layers $V_i$ with $i\equiv 1\pmod{3}$, 
use colors 2 and 3 for all layers $V_i$ with $i\equiv 2\pmod{3}$, 
and
use colors 3 and 1 for all layers $V_i$ with $i\equiv 3\pmod{3}$. 
Then each monochromatic component is contained within two consecutive layers. 
The key to the proof is to control the growth of monochromatic components between consecutive layers. The next lemma is useful for this purpose.

\begin{lemma} \label{enlarge}
Let $w,\Delta,d,k,h\in\mathbb{N}$. 
Let $G$ be a graph with maximum degree at most $\Delta$.
Let $(T,\X)$ be a tree-decomposition of $G$ with width at most $w$, where $\X=(X_t:t\in V(T))$. 
For $i \geq 1$, let $Y_i$ be a subset of $V(T)$, $T_i$ a subtree of $T$ containing $Y_i$, and $E_i$ a set of pairs of vertices in $\bigcup_{x \in Y_i}X_x$ with $\lvert E_i \rvert \leq k$.
Let $G'$ be the graph with $V(G')=V(G)$ and $E(G')=E(G) \cup \bigcup_{i \geq 1} E_i$.
If every vertex of $G$ appears in at most $d$ pairs in $\bigcup_{i \geq 1} E_i$, and every vertex of $T$ is contained in at most $h$ members of $\{T_1,T_2,\dots\}$, then $G'$ has maximum degree at most $\Delta+d$ and has a tree-decomposition $(T,\X')$ of width at most $w+2hk$ with $X'_t \supseteq X_t$ for every $t \in V(T)$, where $\X'=(X'_t: t \in V(T))$.
\end{lemma}

\begin{proof}
Since every vertex of $G$ appears in at most $d$ pairs in $\bigcup_{i \geq 1} E_i$, $G'$ has maximum degree at most $\Delta+d$.
For every $i \geq 1$, let $Z_i$ be the set of the vertices appearing in some pair of $E_i$.
Note that $Z_i \subseteq \bigcup_{x \in Y_i}X_x$ and $\lvert Z_i \rvert \leq 2\lvert E_i \rvert \leq 2k$.
For every $t \in V(T)$, let $X_t' := X_t \cup \bigcup_{\{i: t \in V(T_i)\}}Z_i$.
Let $\X' := (X_t': t \in V(T))$.

We claim that $(T,\X')$ is a tree-decomposition of $G'$.
It is clear that $\bigcup_{t \in V(T)} X'_t \supseteq V(G')$.
For each $i \in\mathbb{N}$, for every $t \in V(T_i)$, since $X_t' \supseteq Z_i$, $X_t'$ contains both ends of each edge in $E_i$.
For each $v \in V(G')$, 
$$\{t \in V(T): v \in X_t'\}=\{t \in V(T): v \in X_t\} \cup \bigcup_{\{i: v \in Z_i\}}V(T_i).$$
Note that for every $v \in V(G)$ and $i \geq 1$, if $v \in Z_i$ then $v\in X_t$ for some  $t \in Y_i \subseteq V(T_i)$.
Hence $\{t: v \in X_t'\}$ induces a subtree of $T$.
This proves that $(T,\X')$ is a tree-decomposition of $G'$.

Since for every $t \in V(T)$, $|X'_t| \leq |X_t| + \sum_{\{i:t\in V(T_i)\}}|Z_i|  \leq w+1+2hk$, the width of $(T,\X')$ is at most $w+2hk$.
Moreover, $X'_t \supseteq X_t$ for every $t \in V(T)$ by definition.
This proves the lemma.
\end{proof}

We now prove \cref{ltwdegree}. 

\begin{theorem}
Let $\Delta, w \in \mathbb{N}$.
Then every graph $G$ with maximum degree at most $\Delta$ and with layered treewidth at most $w$ is 3-colorable with clustering  $g(w,\Delta)$, for some function $g(w,\Delta) \in O(w^{19}\Delta^{37})$.
\end{theorem}

\begin{proof}
Let $f$ be the function from \cref{ClusteringDegreeTreewidth}. Define
\begin{align*}
f_1 & := f(w,\Delta)  & &  \in O(w\Delta)\\
\Delta_2 & :=\Delta+ f_1\Delta^2 & & \in O(w\Delta^3)\\
w_2 & := w+2(w+1)f_1^2\Delta^2 & & \in O(w^3\Delta^4),\\
f_2 & := f(w_2,\Delta_2) & & \in O(w^4\Delta^7)\\
\Delta_3 & := \Delta+f_2\Delta^2 &  & \in O( w^4\Delta^{9})\\
w_3 & := w+4(w_2+1)f_2^2\Delta^2 & & \in O( w^{11} \Delta^{20})\\ 
f_3 & := f(w_3,\Delta_3) & & \in O( w^{15} \Delta^{29} )\\
g(w,\Delta) & :=(1+f_2\Delta)f_3 & & \in O( w^{19} \Delta^{37}  ).
\end{align*}

Let $G$ be a graph of maximum degree at most $\Delta$ and layered treewidth at most $w$.
Let $(T,\X)$ and $(V_i: i \geq 1)$ be a tree-decomposition of $G$ and a layering of $G$ such that $\lvert X_t \cap V_i \rvert \leq w$ for every $t \in V(T)$ and $i \geq 1$, where $\X=(X_t: t \in V(T))$.
For $j \in [3]$, let $U_j = \bigcup_{i=0}^\infty V_{3i+j}$.

By \cref{ClusteringDegreeTreewidth}, there exists a coloring $c_1: U_1 \rightarrow \{1,2\}$ such that every monochromatic component of $G[U_1]$ contains at most $f_1$ vertices.
For each $i \in\mathbb{N}_0$, let $\C_i$ be the set of $c_1$-monochromatic components of $G[U_1]$ contained in $V_{3i+1}$ with color 2.
For each $i \in\mathbb{N}_0$ and $C \in \C_i$, define the following:
	\begin{itemize}
		\item Let $Y_{i,C}$  be a minimal subset of $V(T)$ such that for every edge $e$ of $G$ between $V(C)$ and $N_G(V(C)) \cap V_{3i+2}$, there exists a node $t \in Y_{i,C}$ such that both ends of $e$ belong to $X_t$.
		\item Let $E_{i,C}$  be the set of all pairs of distinct vertices in $N_G(V(C)) \cap V_{3i+2}$.
		\item Let $T_{i,C}$  be the subtree of $T$ induced by $\{t \in V(T): X_t \cap V(C) \neq \emptyset\}$.
	\end{itemize}
Note that there are at most $\lvert V(C) \rvert \Delta \leq f_1\Delta$ edges of $G$ between $V(C)$ and $N_G(V(C))$.
So $\lvert Y_{i,C} \rvert \leq f_1\Delta$ and $\lvert E_{i,C} \rvert \leq f_1^2\Delta^2$ for every $i \in\mathbb{N}_0$ and $C \in \C_i$.
In addition, $Y_{i,C} \subseteq V(T_{i,C})$ for every $i \in\mathbb{N}_0$ and $C \in \C_i$.
Since $(T,(X_t \cap U_1: t \in V(T)))$ is a tree-decomposition of $G[U_1]$ with width at most $w$, for every $t \in V(T)$ and $i \in\mathbb{N}_0$, there exist at most $w+1$ different members $C$ of $\C_i$ such that $t \in V(T_{i,C})$.
Furthermore, $N_G(V(C)) \cap V_{3i+2} \subseteq \bigcup_{x \in Y_{i,C}} X_x$, so each pair in $E_{i,C}$ consists of two vertices in $\bigcup_{x \in Y_{i,C}}X_x$.
Since every vertex $v$ in $U_2$ is adjacent in $G$ to at most $\Delta$ members of $\bigcup_{i \in\mathbb{N}_0}\C_i$ and every member $C$ of $\bigcup_{i \in\mathbb{N}_0}\C_i$ creates at most $f_1\Delta$ pairs in $E_{i_C,C}$ involving $v$, where $i_C$ is the index such that $C \in \C_{i_{C}}$, every vertex in $U_2$ appears in at most $f_1\Delta \cdot \Delta=f_1\Delta^2$ pairs in $\bigcup_{i \in\mathbb{N}_0,C \in \C_i} E_{i,C}$.

Let $G_2$ be the graph with $V(G_2):=U_2$ and 
$$E(G_2):=E(G[U_2]) \cup \bigcup_{i \in\mathbb{N}_0, C \in \C_i} E_{i,C}.$$
We have shown that \cref{enlarge} is applicable with $(G,T,\X)=(G[V_{3i+2}],T,(X_t \cap V_{3i+2}: t \in V(T)))$ (for each $i \in \mathbb{N}_0$), $k=f_1^2\Delta^2$ and $d=f_1\Delta^2$ and $h=w+1$. 
Thus $G_2$ has maximum degree at most $\Delta_2$ and a tree-decomposition $(T,\X^{(2)})$ with $X^{(2)}_t \supseteq X_t$ for every $t \in V(T)$, where $\X^{(2)} = (X^{(2)}_t: t \in V(T))$, such that for each $i \in \mathbb{N}_0$, $(T,(X^{(2)}_t \cap V_{3i+2}: t \in V(T)))$ is a tree-decomposition of $G_2[V_{3i+2}]$ of width at most $w_2$.
Since there exists no edge of $G_2$ between $V_{3i+2}$ and $V_{3i'+2}$ for $i<i'$, by \cref{ClusteringDegreeTreewidth}, there exists a coloring $c_2: U_2 \rightarrow \{2,3\}$ such that every $c_2$-monochromatic component of $G_2$ contains at most $f_2$ vertices.

Since $X^{(2)}_t \supseteq X_t$ for every $t \in V(T)$, we may assume that $(T,\X^{(2)})$ is a tree-decomposition of $G[U_1 \cup U_2]$ by redefining $X^{(2)}_t$ to be the union of $X^{(2)}_t$ and $X_t \cap U_1$, for every $t \in V(T)$.

For each $i \in\mathbb{N}_0$, let $\C_i'$ be the set of the monochromatic components either of $G[U_1]$ with color 1 with respect to $c_1$ contained in $V_{3i+4}$ or of $G_2$ with color 3 with respect to $c_2$ contained in $V_{3i+2}$.
For each $i \in\mathbb{N}_0$ and $C \in \C_i'$, define the following:
	\begin{itemize}
		\item Let $Y'_{i,C}$ be a minimal subset of $V(T)$ such that for every edge $e$ of $G$ between $V(C)$ and $N_G(V(C)) \cap V_{3i+3}$, there exists a node $t \in Y'_{i,C}$ such that both ends of $e$ belong to $X_t$.
		\item Let $E'_{i,C}$ be the set of all pairs of distinct vertices of $N_G(V(C)) \cap V_{3i+3}$.
		\item Let $T'_{i,C}$ be the subtree of $T$ induced by $\{t \in V(T): X^{(2)}_t \cap V(C) \neq \emptyset\}$.
	\end{itemize}
Note that there are at most $\lvert V(C) \rvert \Delta \leq f_2\Delta$ edges of $G$ between $V(C)$ and $N_G(V(C)) \cap V_{3i+3}$ for every $i \in\mathbb{N}_0$ and $C \in \C_i'$.
So $\lvert Y'_{i,C} \rvert \leq f_2\Delta$ and $\lvert E'_{i,C} \rvert \leq f_2^2\Delta^2$ for every $i \in\mathbb{N}_0$ and $C \in \C_i'$.
In addition, $Y'_{i,C} \subseteq V(T'_{i,C})$ for every $i \in\mathbb{N}_0$ and $C \in \C_i'$.
For every $i \in \mathbb{N}_0$, since $(T,(X^{(2)}_t \cap U_2: t \in V(T)))$ is a tree-decomposition of $G_2$ with $\lvert X^{(2)}_t \cap V_{3i+2} \rvert \leq w_2+1$ and $(T,(X_t^{(2)} \cap U_1: t \in V(T)))$ is a tree-decomposition of $G[U_1]$ with $\lvert X^{(2)}_t \cap V_{3i+4} \rvert \leq w+1 \leq w_2+1$, we know that for every $t \in V(T)$, there exist at most $2(w_2+1)$ different members $C \in \C_i$ such that $t \in V(T'_{i,C})$.
Furthermore, each pair in $E'_{i,C}$ consists of two vertices in $\bigcup_{x \in Y'_{i,C}}X_x^{(2)}$.
Since every vertex $v$ in $U_3$ is adjacent in $G$ to at most $\Delta$ members of $\bigcup_{i \in\mathbb{N}_0}\C_i'$, and every member $C$ of $\bigcup_{i \in\mathbb{N}_0}\C_i'$ creates at most $f_2\Delta$ pairs in $E_{i_C,C}$ involving $v$, where $i_C$ is the index such that $C \in \C_{i_C}'$, every vertex in $U_3$ appears in at most $f_2\Delta^2$ pairs in $\bigcup_{i \in\mathbb{N}_0,C \in \C_i'} E_{i,C}$.

Let $G_3$ be the graph with $V(G_3):=U_3$ and 
$$E(G_3):=E(G[U_3]) \cup \bigcup_{i \in\mathbb{N}_0, C \in \C_i'} E'_{i,C}.$$
We have shown that \cref{enlarge} is applicable with $(G,T,\X)=(G[V_{3i+3}],T,(X_t \cap V_{3i+3}: t \in V(T)))$ (for each $i \in \mathbb{N}_0$), $k=f_2^2\Delta^2$ and $d=f_2\Delta^2$ and $h=2(w_2+1)$. 
Hence $G_3$ has maximum degree at most $\Delta_3$ and a tree-decomposition $(T,\X^{(3)})$ with $X^{(3)}_t \supseteq X_t$ for each $t \in V(T)$, where $\X^{(3)}=(X^{(3)}_t: t \in V(T))$, such that $(T,(X^{(3)}_t \cap V_{3i+3}: t \in V(T)))$ is a tree-decomposition of $G_3[V_{3i+3}]$ of width at most $w_3$ for every $i \in \mathbb{N}_0$.
Since there exists no edge of $G_3$ between $V_{3i+3}$ and $V_{3i'+3}$ for $i<i'$, by \cref{ClusteringDegreeTreewidth}, there exists a coloring $c_3: U_3 \rightarrow \{1,3\}$ such that every monochromatic component of $G_3$ contains at most $f_3$ vertices.

Define $c: V(G) \rightarrow \{1,2,3\}$ such that for every $v \in V(G)$, we have $c(v):=c_j(v)$, where $j$ is the index for which $v \in U_j$.
Now we prove that every $c$-monochromatic component of $G$ contains at most $g(w,\Delta)$ vertices.

Let $D$ be a $c$-monochromatic component of $G$ with color 2.
Since $D$ is connected, for every pair of vertices $u,v \in V(D) \cap U_2$, there exists a path $P_{uv}$ in $D$ from $u$ to $v$.
Since $V(D) \subseteq U_1 \cup U_2$, for every maximal subpath $P$ of $P_{uv}$ contained in $U_1$, 
there exists $i\in\mathbb{N}_0$ and $C \in \C_i$ such that there exists a pair in 
$E_{i,C}$ consisting of the two vertices in $P_{uv}$ adjacent to the ends of $P$.
That is, there exists a path in $G_2[V(D) \cap U_2]$ connecting $u,v$ for every pair $u,v \in V(D) \cap U_2$.
Hence $G_2[V(D) \cap U_2]$ is connected.
So $G_2[V(D) \cap U_2]$ is a $c_2$-monochromatic component of $G_2$ with color 2 and contains at most $f_2$ vertices.
Hence there are at most $f_2\Delta$ edges of $G$ between $V(D) \cap U_2$ and $N_G(V(D) \cap U_2)$.
So $D[V(D) \cap U_1]$ contains at most $f_2\Delta$ components.
Since each component of $D[V(D) \cap U_1]$ is a $c_1$-monochromatic component of $G[U_1]$, it contains at most $f_1$ vertices.
Hence $D[V(D) \cap U_1]$ contains at most $f_2\Delta f_1$ vertices.
Since $D$ has color 2, $V(D) \cap U_3=\emptyset$.
Therefore, $D$ contains at most $(1+f_1\Delta)f_2 \leq g(w,\Delta)$ vertices.

Let $D'$ be a $c$-monochromatic component of $G$ with color $b$, where $b \in \{1,3\}$.
Since $D'$ is connected, by an analogous argument to that in the previous paragraph, $G_3[V(D') \cap U_3]$ is connected. 
So $G_3[V(D') \cap U_3]$ is a $c_3$-monochromatic component of $G_3$ with color $b$ and contains at most $f_3$ vertices.
Hence there are at most $f_3\Delta$ edges of $G$ between $V(D') \cap U_3$ and $N_G(V(D') \cap U_3)$.
So $D[V(D) \cap U_{b'}]$ contains at most $f_3\Delta$ components, where $b'=1$ if $b=1$ and $b'=2$ if $b=3$.
Since each component of $D[V(D) \cap U_{b'}]$ is a $c_{b'}$-monochromatic component of $G[U_{b'}]$, it contains at most $f_2$ vertices.
Hence $D[V(D) \cap U_{b'}]$ contains at most $f_3\Delta f_2$ vertices.
Since $D$ has color $b$, $V(D) \cap U_{b+1}=\emptyset$, where $U_4=U_1$.
Therefore, $D$ contains at most $(1+f_2\Delta)f_3 \leq g(w,\Delta)$ vertices.
This completes the proof.
\end{proof}

\section{Subsequent Work}

The following extension of \cref{ltwdegree} is proved in our companion paper~\citep{LW1}. 

\begin{theorem}[\citep{LW1}]
\label{ltwbasic}
For all $s,t,w\in\mathbb{N}$ there exists $\eta\in\mathbb{N}$ such that every graph with layered treewidth at most $w$ and with no $K_{s,t}$ subgraph is $(s+2)$-colorable with clustering $\eta$. 
\end{theorem}

Graphs with no $K_{1,t}$ subgraph are exactly those with maximum degree less than $t$. Thus, \cref{ltwbasic} with $s=1$ says that graphs with bounded layered treewidth and bounded maximum degree are 3-colourable with bounded clustering, as in \cref{ltwdegree}. On the other hand, the proof of \cref{ltwbasic} is much more complicated than that of \cref{ltwdegree} and the clustering bound in \cref{ltwbasic} is much larger than the polynomial clustering bound in \cref{ltwdegree}.

\cref{ltwdegree,ltwbasic} (and their proofs) were first presented in reference~\citep{LW1} in 2019. This paper has now been split into the present paper and a separate paper proving \cref{ltwbasic}. Other proofs of \cref{ltwdegree} have since been obtained. One such proof is obtained via weak diameter coloring~\citep{BBEGLPS,Liu21a}, although this proof does not give a polynomial bound for the clustering. Another proof, due to \citet{DEMWW22}, improves the $O(w^{19}\Delta^{37})$ bound in \cref{ltwdegree} to $O(w^3\Delta^2)$.

Recall that \citet{LO17} proved that every $K_t$-minor free graph with bounded maximum degree is 3-colorable with bounded clustering. New proofs of this result are included in \citep{BBEGLPS,Liu21a,Liu22}, where the proof in \citep{Liu22} relies on \cref{ltwdegree} in the present paper.

\bibliographystyle{DavidNatbibStyle}
\bibliography{DavidBibliography}
\end{document}